\documentclass[11pt]{article}
\usepackage[usenames]{color}
\usepackage{amsmath,amsfonts,amsthm}
\usepackage{syntonly}
\usepackage{setspace}
\usepackage[mathcal]{euscript}
\usepackage{xypic}
\usepackage[all]{xy}
\usepackage{color}   

\def\ZZ{{\mathbb Z}}

\def\cG{{\cal G}}

\def\cT{{\cal T}}

\def\cV{{\cal V}}

\def\cK{{\cal K}}

\def\Ext{{\hbox{\rm{Ext}}}}
\def\Ass{{\hbox{\rm {Ass}}}}
\def\Hom{{\hbox{\rm{Hom}}}}
\def\Dim{{\hbox{\rm{dim}}}}

\def\Im{{\hbox{\rm{Im}}}}
\def\Ann{{\hbox{\rm{Ann}}}}
\def\InjDim{{\hbox{\rm{inj.dim}}}}
\def\Supp{{\hbox{\rm{Supp}}}}
\def\Spec{{\hbox{\rm{Spec}}}}

\def\Length{{\hbox{\rm{length}}}}

\newtheorem{Def}{Definition}[section]
\newtheorem{Lemma}{Lemma}[section]
\newtheorem{Prop}{Proposition}[section]
\newtheorem{Cor}{Corollary}[section]
\newtheorem{Teo}{Theorem}[section]

\begin{document}

\title{Local cohomology properties of direct summands}

\author{Luis N\'u\~nez-Betancourt}
\maketitle

\begin{abstract}
In this article, we prove that
if $R\to S$ is a homomorphism of Noetherian rings that splits, then for every $i\geq 0$ and ideal $I\subset R$, $\Ass_R H^i_I(R)$ is finite when 
$\Ass_S H^i_{IS}(S)$ is finite.
In addition,  if $S$ is a Cohen-Macaulay ring that is finitely generated as an $R$-module, such that
all the Bass numbers of $H^i_{IS}(S)$, as an $S$-module, are finite, then all the Bass numbers of $H^i_{I}(R)$, as  an $R$-module, are finite. Moreover, we show these results for a larger class a functors introduced by Lyubeznik \cite{LyuDMod}. As a consequence, we exhibit a Gorenstein
$F$-regular UFD of positive characteristic that is not a direct summand, not even a pure subring, of any regular ring.
\end{abstract}

\section{Introduction}

Throughout this article all rings are commutative Noetherian with unity. Let $R$ denote a ring. 
If $M$ is an $R$-module and $I\subset R$ is an ideal, we denote the $i$-th local
cohomology of $M$ with support in $I$ by $H^i_I (M)$. If $I$ is generated by the elements $f_1,\ldots, f_\ell \in R$, these cohomology
 groups can be computed by the $\check{\mbox{C}}$ech complex,

$$
0\to M\to \oplus_j M_{f_j}\to \ldots \to M_{f_1 \cdots f_\ell} \to 0.
$$

The structure of these modules has been widely studied 
by several authors. Among the results obtained, one encounters the following finiteness properties for certain 
regular rings:

\begin{itemize}
\item[(1)] \quad the set of associated primes of $H^i_I(R)$ is finite;
\item[(2)] \quad the Bass numbers of $H^i_I(R)$ are finite;
\item[(3)] \quad $\InjDim H^i_I (R)\leq \Dim\Supp H^i_I(R)$.
\end{itemize}

Huneke and Sharp proved those properties for characteristic $p>0$  \cite{Huneke}. Lyubeznik showed  them
for regular local rings of equal characteristic zero  and finitely generated regular algebras over a field of characteristic zero
 \cite{LyuDMod}.

These properties have been proved for a larger family of functors introduced by Lyubeznik  \cite{LyuDMod}.
If $Z \subset \Spec ( R)$ is a closed subset and $M$ is an $R$-module, we denote by 
$H^i_Z (R)$
the $i$-th local cohomology module of $M$ with support in $Z$. We notice that
$H^i_Z (R)=H^i_I (R)$,
where $Z=\cV(I)=\{P\in\Spec(R): I\subset P\}$.
For two
closed subsets of $\Spec(R)$, $Z_1\subset Z_2$, there is a long exact sequence of functors
\begin{equation}\label{LC2}
\ldots\to H^i_{Z_1}\to H^i_{Z_2}\to H^i_{Z_1/Z_2}\to \ldots
\end{equation}
We denote by $\cT$ any functor of the form $\cT =\cT_1\circ \dots\circ\cT_t$, where every functor $\cT_j$
is either $H^i_Z$ for some closed subset $Z$ of $\Spec (R)$
or the kernel, image or cokernel of some morphism in the previous long exact sequence
for some closed subsets $Z_1,Z_2$ of $\Spec(R)$.

Our aim in this manuscript is to prove the finiteness properties ($1$) and ($2$) for direct summands. We need to make some observations before we are able
to state our theorems precisely. Let $R\to S$ be a homomorphism of Noetherian rings. For an ideal $I\subset R$,
we have two functors associated with it, $H^i_{I}(-): R\hbox{-mod}\to 
R\hbox{-mod}$ 
and $H^i_{IS}(-): S\hbox{-mod}\to S\hbox{-mod}$, which are naturally isomorphic when we restrict them to $S$-modules.
Moreover, for two ideals of $R$, $I_2\subset I_1$, 
the natural morphism $H^i_{I_1}(-)\to H^i_{I_2}(-)$ is the same 
as the natural morphism $H^i_{I_1 S}(-)\to H^i_{I_2 S}(-)$ when we restrict the functors 
to $S$-modules. Thus, their kernel, cokernel and image are naturally isomorphic as
$S$-modules.  Hence, every functor $\cT$ for $R$ is a functor of the same type for
$S$ when we restrict it to $S$-modules.

As per the previous discussion, for an $S$-module, $M$, we will make no distinction in the notation or 
meaning of $\cT(M)$ whether it is
induced by ideals of $R$ or their extensions to $S$ and, therefore, by the corresponding closed subsets of their respective spectra.
Now, we are ready to state our main results.

\begin{Teo}\label{MainThm1}
Let $R\to S$ be a homomorphism of Noetherian rings that splits. Suppose that  
$\Ass_S\cT(S)$ is finite for a functor $\cT$ induced by extension of ideals of $R$.
Then, $\Ass_R\cT(R)$ is finite. 
In particular,  $\Ass_R H^i_I(R)$ is finite for every ideal $I\subset R$,
if $\Ass_S H^i_{IS}(S)$ is finite.
\end{Teo}

\begin{Teo}\label{MainThm2}
Let $R\to S$ be a homomorphism of Noetherian rings that splits. 
Suppose that $S$ is a Cohen-Macaulay ring such that   
all the Bass numbers of $\cT(S)$, as an $S$-module, are finite for a functor $\cT$ induced by extension of ideals of $R$.
If $S$ is a finitely 
generated $R$-module, then all the Bass numbers of $\cT(R)$, as an $R$-module, are finite. In particular,  for every ideal $I\subset R$
the Bass numbers of $ H^i_I(R)$ are finite,
if the Bass numbers of $H^i_{IS}(S)$ are finite.
\end{Teo}

The first theorem holds when 
$S$ is a polynomial ring over a field and $R$ is the invariant ring of an 
action of a linearly reductive group over $S$. It also holds when $R\subset K[x_1,\ldots,x_n]$ is an integrally closed 
ring that is finitely generated as a $K$-algebra by monomials. This is because such a ring is a direct summand 
of a possibly different polynomial ring (cf. Proposition $1$ and Lemma $1$ in \cite{MelMonomials}).

We would like to mention another case in which an inclusion splits.
This is when $R\to S$ is a module finite 
extension of rings containing a field of characteristic zero such that $S$ has finite projective dimension as an $R$-module. Moreover, such a
splitting exists when Koh's conjecture holds (cf. \cite{Koh, Velez1, Velez2}). 
Therefore, if Koh's conjecture applies to $R\to S$ and
$\cT(S)$ has finite associated primes or finite Bass numbers, so does $\cT(R)$.

We point out that property ($3$) does not hold for direct summands of regular rings, even in the finite extension case. 
A counterexample is 
$R=K[x^3,x^2 y, xy^2,y^3]\subset S=K[x,y]$, where $S$ is the polynomial ring in two variables with coefficients
in a field $K$. The splitting of the inclusion is the map $\theta :S\to R$ defined in the monomials by 
$\theta(x^\alpha y^\beta)=x^\alpha y^\beta$ if $\alpha + \beta \in 3\ZZ$ and as zero otherwise. We have that the dimension of 
$\Supp (H^2_{(x^3,x^2 y, xy^2,y^3)} (R))$  is zero, but it is not an injective module, because $R$ is not a Gorenstein ring, 
since $R/(x^3,y^3)R$ has a two dimensional socle.  

The manuscript is organized as follows. In section $2$, we prove Theorem \ref{MainThm1}, and we show some
consequences. In particular, we exhibit a Gorenstein
$F$-regular UFD of positive characteristic that is not a direct summand, not even a pure subring,  of any regular ring. 
 In section 3, we give a proof for Theorem \ref{MainThm2}.
 

\section{Associated Primes}

\begin{Lemma}\label{PropAss} Let $R\to S$ be an injective homomorphism of Noetherian rings, and let $M$ 
be an $S$-module. Then, $\Ass_R M \subset \{ Q\cap R : Q\in \Ass_S M\}$.
\end{Lemma}
\begin{proof}
Let $P\in \Ass_R M$ and $u\in M$ be such that $\Ann_R u=P$. We have that $(\Ann_S u)\cap R= P$. 
Let $Q_1,\ldots,Q_t$ denote
the minimal primes of $\Ann_S u$. We obtain that
$$
P=\sqrt{P}=\sqrt{\Ann_S u} \cap R =(\cap_{j} Q_j ) \cap R=\cap_{j} (Q_j \cap R),  
$$
so, there exists a $ Q_j$ such that $P=  Q_j\cap R$. Since $Q_j$ is a minimal prime for 
$\Ann_S u$, we have that $Q_j \in \Ass_S M$ and the result follows.
\end{proof}
\begin{Def}
We say that a homomorphism of Noetherian rings $R \to S$ is pure if  $M = M \otimes_R R \to M \otimes_R S$ 
is injective for every $R$-module $M.$ We also say that $R$ is a pure subring of $S$. 
\end{Def}
\begin{Prop} [Cor. $6.6$ in \cite{HoRo}]\label{PropMel}
Suppose that $R \to S$ is a pure homeomorphism of Noetherian rings and that $ \cG$ is a complex
of $R$-modules. Then, the induced map
$j: H^i(\cG) \to H^i(\cG \otimes_R S)$
is injective.
\end{Prop}
\begin{Prop}
Let $R\to S$ be a pure homomorphism of Noetherian rings. Suppose that  
$\Ass_S H^i_{I}(R)$ is finite for some ideal $I\subset R$ and  $i\geq 0$. Then,
$\Ass_S H^i_{IS}(S)$ is finite.
\end{Prop}
\begin{proof}
Since $H^i_{I}(R)\to H^i_{IS}(S)$ is injective by Proposition \ref{PropMel}, $\Ass_R H^i_{I}(R)\subset\Ass_R H^i_{IS}(S)$ and the result follows by Lemma \ref{PropAss}. 
\end{proof}
\begin{proof}[Proof of Theorem \ref{MainThm1}]
The splitting between $R$ and $S$ makes
 $\cT (R)$ into a direct summand of  $\cT (S)$; in particular, $\cT (R)\subset\cT (S)$.  
Therefore, $\Ass_R \cT(R)\subset\Ass_R \cT(S)$ and the result follows by Lemma \ref{PropAss}. 
\end{proof}

If $R$ is a ring containing a field of characteristic $p>0$,  Theorem \ref{MainThm1} gives a method for 
showing that $R$ is not a direct summand of a regular ring. We used this method to prove that there exists a Gorenstein strongly 
$F$-regular UFD of characteristic $p>0$ that is not a direct summand of any regular ring.

\begin{Teo}[Thm. 5.4 in \cite{Anu}] Let K be a field, and consider the hypersurface
$$
R=\frac{K[r, s, t, u, v, w, x, y, z]}{(su^2x^2+ sv^2y^2+ tuxvy + rw^2z^2)}.
$$
Then, $R$ is a unique factorization domain for which the local cohomology module
$H^3_{(x,y,z)}(R)$ has infinitely many associated prime ideals. This is preserved if $R$ is replaced
by the localization at its homogeneous maximal ideal. The hypersurface $R$ has rational
singularities if $K$ has characteristic zero, and it is $F$-regular if $K$ has positive characteristic.
\end{Teo}
\begin{Cor} Let $R$ be as in the previous theorem taking $K$ of positive characteristic. Then, $R$ is a Gorenstein
$F$-regular UFD that is not a pure subring of any regular ring. In particular, $R$ is not direct summand of any regular ring.
\end{Cor}
\begin{proof}
Since $H^3_{(x,y,z)}(R)$ has infinitely many associated prime ideals, it cannot be a direct summand or pure subring 
of a regular ring by
Theorem \ref{MainThm1}, Proposition \ref{PropMel} and finiteness properties of regular rings 
of positive characteristic (cf. \cite{LyuFMod}).
\end{proof}
\begin{Teo}[Thm. 1 in \cite{Zhang}] Assume  that $S=K[x_1,\ldots,x_n]$ is a polynomial ring in $n$ variables over a 
field $K$ of characteristic $p>0$. 
Suppose that $I = (f_1,\ldots, f_s)$ is an ideal of $S$ such that 
$\sum_i \deg f_i < n$.
Then $\dim S/Q \geq n-\sum_i \deg f_i$ 
for all $Q\in\Ass_S H^i_I(S)$.
\end{Teo}
\begin{Cor}
Let  $S=K[x_1,\ldots,x_n]$ be a polynomial ring in $n$ variables over a field $K$ of characteristic $p>0$.
Let $R\to S$ be a homomorphism of Noetherian rings that splits.
Suppose that $I = (f_1,\ldots, f_s)$ is an ideal of $R$ such that $\sum_i \deg (f_i) < \Dim{R}$.
If  $S$ is a finitely generated $R$-module,
then $\dim R/P \geq \dim R-\sum_i \deg f_i$ for all $P\in\Ass_R H^i_I(R)$.
\end{Cor}
\begin{proof}
Since $H^i_I (-)$ commutes with direct sum of $R$-modules, we have that a splitting of $R\hookrightarrow S$ over $R$  induces an splitting of
$H^i_I (R) \hookrightarrow H^i_I (S)$ over $R$. Then, by Lemma \ref{PropAss}, for any $P\in\Ass_R H^i_I(R)\subset \Ass_R H^i_I(S)$ there exists
$Q\in\Ass_R H^i_I(S)$ such that $P=Q\cap R$ and then $\Dim R/P=\Dim S/Q > n-\sum_i \deg f_i$, and the result follows.
\end{proof}

\section{Bass Numbers}
\begin{Lemma}\label{LemmaLength}
Let $(R,m,K)$ be a local ring and $M$ be an $R$-module. Then, the following are equivalent:
\begin{itemize} 
\item[a)] $\Dim_K(\Ext^j_R(K,M))$ is finite for all $j\geq 0$;
\item[b)] $\Length(\Ext^j_R(N,M))$ is finite for every finite length module $N$ for all $j\geq 0$;
\item[c)] there exists one module $N$ of finite length such that $\Length(\Ext^j_R(N,M))$ is finite for all $j\geq 0$.
\end{itemize}
\end{Lemma}
\begin{proof} \emph{ a) $\Rightarrow$ b): }Our proof will be by induction on $h=\Length(N)$.
If $h=1$, then $N=K$, and the proof follows from our assumption.   We will assume that the statement is true for $h$ and prove it 
when $\Length (N)=h+1$. In this case, there is a short exact sequence
$0\to K\to N\to N'\to 0$, where $N'$ has length $h$. From the induced long exact sequence  
$$\ldots \to\Ext^{j-1}_R(N',M)\to \Ext^j_R(K,M)\to \Ext^j_R(N,M)\to\ldots,$$
we see that $\Length(\Ext^i_R(N,M))$ is finite for all $i\geq 0$. \\

\emph{ b) $\Rightarrow$ c): } Clear.\\

\emph{c) $\Rightarrow$ a): } We will prove the contrapositive. 
Let $j$ be the minimum non-negative integer such that $\Dim_K (\Ext^j_R(K,M))$ is infinite. We claim
that $\Length(\Ext^i_R(N,M))<\infty $ for $i<j$ and $\Length(\Ext^j_R(N,M))=\infty$ for any module $N$ of finite length. 
Our proof will be by induction on $h=\Length(N)$.
If $h=1$, then $N=K$ and it follows from our choice of $j$. We will assume that this is true for $h$ and prove it 
when $\Length (N)=h+1$. We have a short exact sequence
$0\to K\to N\to N'\to 0$, where $N'$ has length $h$. From the induced  long exact sequence  
$$\ldots \to\Ext^{j-1}_R(N',M)\to \Ext^j_R(K,M)\to \Ext^j_R(N,M)\to\ldots,$$
we have that $\Length(\Ext^i_R(N,M))<\infty$ for $i<j$ and that the map \\$\Ext^j_R(K,M)/\Im(\Ext^{j-1}_R(N',M))\to \Ext^j_R(N,M)$ is injective. Therefore, \\$\Length(\Ext^j_R(N,M))=\infty$.
\end{proof}

\begin{Lemma}\label{LemmaHyp}
Let $R\to S$ be a pure homomorphism of Noetherian rings. Assume that $S$ is a Cohen-Macaulay ring.
If $S$ is finitely generated as an $R$-module, then $R$ is a Cohen-Macaulay ring.
\end{Lemma}
\begin{proof}
Let $P\subset R$ be a prime ideal. Let $\underline{x}=x_1,\ldots, x_d$ denote a system of parameters of $R_P$, where $d=\dim(R_P)$.
It suffices to show that $H_i(\cK(\underline{x};R_P))=0$ for $i\neq 0$,  where $\cK$ is the
Koszul complex with respect to $\underline{x}.$
We notice that the natural inclusion $R_P\to S_P$ is a pure homeomorphism of rings.
This induces an injective morphism of $R$-modules
$H_i(\cK(\underline{x};R_P))\to H_i(\cK(\underline{x};S_P))$ by Proposition \ref{PropMel}. Thus, it is enough to show that $H_i(\cK(\underline{x};S_P))=0$ for $i\neq 0$.
Since $S_P$ is a module finite extension of $R_P$, we have that  every maximal ideal $Q\subset S_P$ contracts to $PR_P$ and 
$\underline{x}$ is a system of parameters for $S_Q$. Then, $H_i(\cK(\underline{x};S_Q))=0$ for $i\neq 0$ and every maximal ideal $Q\subset S_P$.
Hence, $H_i(\cK(\underline{x};S_P))=0$ for $i\neq 0$ and the result follows.

\end{proof}

\begin{Prop}\label{PropBass} 
Let $R\to S$ be a homomorphism of  Noetherian rings that splits. 
Assume that $S$ is a Cohen-Macaulay ring and  $S$ is finitely generated as an $R$-module.
Let $N$ be an $R$-module and $M$ be an $S$-module. Let $N\to M$ be a morphism of $R$-modules that splits. 
If all the Bass numbers of $M$, as an $S$-module, are finite, 
then all the Bass numbers of $N$, as an $R$-module, are finite.
\end{Prop}
\begin{proof}

Since $N\hookrightarrow  M$ splits, we have that $\Ext^i_{R_P} (R_P/PR_P,N_P)$ 
is a direct summand of $\Ext^i_{R_P} (R_P/PR_P,M_P)$,
so, we may assume that $N=M$.

Let $P$ be a fixed prime ideal of $R$ and let $K_P$ denote $R_P/PR_P$. Since we want to show that
$\Dim_{K_P}(\Ext^i_{R_P} (K_P, M_P))$ is finite, we may assume without loss of generality that $R$ is local and $P$ is
its maximal ideal. Let $\underline{x}=x_1,\dots,x_n$ be a system of parameters for $R$. 
Since $R$ is Cohen-Macaulay by Lemma \ref{LemmaHyp}, we have that the 
Koszul complex, $\cK_R(\underline{x})$, 
is a free resolution for $R/I$, where $I=(x_1,\ldots,x_n)$. We also have that for every maximal ideal 
$Q\subset S$ lying over $P$, $\underline{x}$ is a system of  parameters of $S_Q$ because $\Dim R=\Dim S_Q$ and $S_Q/IS_Q$ 
is a zero dimensional ring. 
From the Cohen-Macaulayness of $S$ and the previous fact, we have that the Koszul complex $\cK_S (\underline{x})$ is a free
resolution for $S/IS$. Therefore, 
$\Ext^i_R (R/I, M)=H^i(\Hom_R(\cK_R (\underline{x}), M))=H^i(\Hom_S(\cK_S (\underline{x}),M))=\Ext^i_S (S/IS, M).$
Since $\Ext^i_S (S/IS, M)=\oplus_Q \Ext^i_{S_Q} (S_Q/IS_Q, M_Q)$ has finite length as an $S$-module by Lemma \ref{LemmaLength}, we have that 
$\Ext^i_R (R/I, M)$ has finite length as an $R$-module because $S$ is finitely generated.
Then, we have that $\Dim_{K_P} (\Ext^i_R (K_P, M))$ is finite by Lemma \ref{LemmaLength}.\end{proof}

\begin{proof}[Proof of Theorem \ref{MainThm2}]
The splitting between $R$ and $S$ induces a splitting between
 $\cT (R) \hookrightarrow \cT (S)$.  
The rest follows from  Proposition \ref{PropBass}.
\end{proof}

\section*{Acknowledgments}
I would like to thank my advisor Mel Hochster for his valuable comments and suggestions. 
I also wish to thank Juan Felipe Perez-Vallejo for carefully reading 
this manuscript. 
I am grateful to the
referee for her or his comments.
Thanks are also due to the National Council of Science and Technology of Mexico by its support through grant $210916.$

{\sc Department of Mathematics, University of Michigan, Ann Arbor, MI $48109$--$1043$, USA.}\\
{\it Email address:}  \texttt{luisnub@umich.edu}
\end{document}